\documentclass[11pt, article]{amsart}

\usepackage{color}

\usepackage{graphicx, epstopdf}
\usepackage{cases}

\usepackage{enumerate}
\usepackage{bbm}
\usepackage{amssymb}
\usepackage{amscd}
\usepackage{amsfonts,latexsym,amsmath,amsxtra,mathdots,amssymb,latexsym,mathtools}
\usepackage{mathabx}
\usepackage[all,cmtip]{xy}

\usepackage{colordvi}
\usepackage{multicol}
\usepackage{hyperref,dirtytalk}
\usepackage{tikz-cd}
\usepackage{float}
\usepackage{setspace, floatflt}
\usepackage{tensor}
\usepackage[normalem]{ulem}

\allowdisplaybreaks

\newcommand{\BA}{{\mathbb {A}}} \newcommand{\BC}{{\mathbb {C}}}

 \newcommand{\BN}{{\mathbb {N}}} 
 \newcommand{\BR}{{\mathbb {R}}} 
  
 \newcommand{\BZ}{{\mathbb {Z}}} 

\newcommand{\CA}{{\mathcal {A}}} \newcommand{\CC}{{\mathcal {C}}} 
 
\newcommand{\CE}{{\mathcal {E}}} 
 \newcommand{\CL}{{\mathcal {L}}}
\newcommand{\CM}{{\mathcal {M}}} \newcommand{\CP}{{\mathcal {P}}}
 
 \newcommand{\CW}{{\mathcal {W}}}
 
\newcommand{\GL}{{\mathrm {GL}}} 
\newcommand{\SL}{{\mathrm {SL}}} 
\newcommand{\SO}{{\mathrm{SO}}}

 \newcommand{\Tr}{{\mathrm{Tr}}}

\newcommand{\Irr}{{\mathrm{Irr}}}
\newcommand{\Htemp}{\mathrm{Irr}_{\mathrm{H-temp}}}
\newcommand{\Res}{{\mathrm{Res}}}
\newcommand{\shal}{\mathrm{Irr}_{\mathrm{Shal}}}

\newcommand{\udist}{\mathrm{Irr}_{\mathrm{unit, dist}}}
\newcommand{\ugdist}{\mathrm{Irr}_{\mathrm{unit, gen, dist}}}
\newcommand{\dist}{\mathrm{Irr}_{\mathrm{dist}}}
\newcommand{\gdist}{\mathrm{Irr}_{\mathrm{gen,dist}}}
\newcommand{\usqdist}{\mathrm{Irr}_{\mathrm{sq,dist}}}
\newcommand{\tdist}{\mathrm{Irr}_{\mathrm{temp,dist}}}
\newcommand{\temp}{\mathrm{Irr}_{\mathrm{temp}}}
\newcommand{\uni}{\mathrm{Irr}_{\mathrm{unit}}}
\newcommand{\gen}{\mathrm{Irr}_{\mathrm{gen}}}

\newcommand{\ugen}{\mathrm{Irr}_{\mathrm{unit,gen}}}

\newcommand{\tshal}{\mathrm{Irr}_{\mathrm{temp,Shal}}}

\newcommand{\gshal}{\mathrm{Irr}_{\mathrm{gen,Shal}}}
\newcommand{\sm}{\mathcal{C}^{\infty}}

\newcommand{\Hom}{\mathrm{Hom}}
\newcommand{\ind}{\mathrm{ind}} \newcommand{\Ind}{\mathrm{Ind}} 
\newcommand{\Ad}{\mathrm{Ad}}
\newcommand{\triv}{{\mathbf{1}}}
\newcommand{\usq}{\mathrm{Irr}_{\mathrm{sq}}}

\def\-{^{-1}}

\def\1{\mathbf{1}}

\def\d{\delta}
\def\e{\varepsilon}

\def\St{\mathrm{St}}

\def\diag{\mathrm{diag}}

\renewcommand{\Re}{{\mathrm{Re}\,}}

\makeatletter
\g@addto@macro\normalsize{\setlength\abovedisplayskip{3pt}}
\makeatother

\makeatletter
\g@addto@macro\normalsize{\setlength\belowdisplayskip{3pt}}
\makeatother

\newcommand{\delete}[1]{}

\theoremstyle{plain}

\newtheorem{thm}{Theorem}[section] 
\newtheorem{cor}[thm]{Corollary}

\newtheorem*{st*}{Statement}
\newtheorem*{q*}{Question}
\newtheorem*{thm*}{Theorem}
\newtheorem*{lm*}{Lemma}
\newtheorem*{cor*}{Corollary}
\newtheorem*{rem*}{Remark}
\newtheorem{lem}[thm]{Lemma}  \newtheorem{prop}[thm]{Proposition}

\newtheorem {rem}[thm]{Remark}

\numberwithin{equation}{section}

\begin{document}

	\title[Gamma factors, root numbers, and distinction]{Gamma factors and root numbers of pairs for the Galois and the linear model}
	
	\author{N. Matringe}
	\address{Nadir Matringe. Institute of Mathematical Sciences, NYU Shanghai, 3663 Zhongshan Road North Shanghai, 200062, China and
		Institut de Math\'ematiques de Jussieu-Paris Rive Gauche, Universit\'e Paris Cit\'e, 75205, Paris, France}
	\email{nrm6864@nyu.edu and matringe@img-prg.fr}

	\maketitle
	
\vspace{-0.5cm}	
	
\begin{abstract} 
Using harmonic analysis on Harish-Chandra Schwartz spaces of various spherical spaces, we extend a relative local converse theorem of Youngbin Ok for the Galois model of $p$-adic $\GL_n$, from the class of cuspidal representations to that of square integrable representations, which is not its optimal form. We also prove a variant of this result for linear models by the same method. The above statements are luckily non empty as we verify triviality results for gamma and epsilon factors of pairs of distinguished representations at the central value $s=1/2$. Along the way, we offer a new proof of conjectures of D. Prasad and D. Ramakrishnan on local components of symplectic cuspidal automorphic representations, and root numbers of pairs of symplectic representations. 
\end{abstract}

\section{Introduction}

Let $F$ be a non Archimedean local field of characteristic zero, and $E/F$ a field extension of degree at most $2$. For $n\geq 1$ set $G_n:=\GL_n(E)$, and denote by $\theta$ the involution of $G_n$ induced by the Galois involution of $E/F$ when $[E:F]=2$, and the conjugation by the matrix 
$\diag(1,-1,\dots,(-1)^{n-1})$ when $E=F$. Let $H_n=G_n^\theta$ be the subgroup of $G_n$ fixed by $\theta$, and $\psi:E\to \BC^\times$ be a non trivial character of $E$, trivial on $F$ if $[E:F]=2$. We recall that a complex smooth irreducible representation of $G_n$ is called distinguished if it has a nonzero $H_n$-invariant linear form on its space, in which case this linear form lives in a one dimensional space by \cite[Proposition 11]{Fli} and \cite[Theorem 1.1]{JR}. We prove as the main result of this paper the following result (see Theorem \ref{thm main}). 

\begin{thm*}[RLC(n,n-1)]
Let $a:=\frac{2}{[E:F]}\in \{1,2\}$, $\psi:E\to \BC^\times$ be a non trivial character of $E$ trivial on $F$ when $[E:F]=2$, and $\pi$ be a square integrable representation of $G_{an}$ for $n\geq 2$. If the Rankin-Selberg central gamma value $\gamma(1/2,\pi,\pi',\psi)$ is equal to one for any $H_{a(n-1)}$-distinguished tempered representation $\pi'$ of $G_{a(n-1)}$, then $\pi$ is $H_{an}$-distinguished.
\end{thm*}

 We refer to the above statement as RLC(n,n-1), with RLC for relative local converse, because if one allows $E=F\times F$ (the so called group case), this statement transforms into the usual local converse theorem of Henniart in \cite{H93} for square integrable representations. More generally RLC(n,m) refers to the statement of the theorem, where one twists by $H_{am}$-distinguished tempered representations. The statement of RLC(n,n-1) is not empty as we check in Corollary \ref{cor trivial gamma} that $\gamma(1/2,\pi,\pi',\psi)$ is equal to one for any $H_{an}$-distinguished generic unitary representation $\pi$ of $G_{a_n}$ and any $H_{am}$-distinguished tempered representation $\pi'$ of $G_{am}$ whenever $n,m\geq 1$. Actually if one considers the root number instead, we check in Theorem \ref{thm trivial local root} that it is equal to one for any pair $(\pi,\pi')$ of irreducible distinguished representations of $G_{an}$ and $G_{am}$.

When $E=F$ our main result is completely new. On the other hand when $[E:F]=2$ it has a longer history which we recall. The statement RLC(2,1) was first obtained in the pioneering work \cite{Hduke} for $\GL_2(E)$. Hakim actually proves that the theorem holds for any generic unitary representation $\pi$ of $\GL_2(E)$, and its statement was extended to all generic representations in \cite{MatPac}. For $\GL_3(E)$, it is a consequence of \cite[Theorem 1.7.1]{HO} that RLC(3,2) holds for any generic unitary representation. Note that these works allow twists by generic unitary representations, but as such representations always have enough tempered representations in their connected component, tempered twists are sufficient by meromorphy of gamma factors, as was observed in \cite[Lemma 5.4]{MO}. 
Following Hakim's work on $\GL_2(E)$, it was proved by Ok in \cite[Main Theorem III]{Ok} that RLC(n,n-1) holds when $\pi$ is cuspidal and $[E:F]=2$. One fundamental input of Ok's proof is Bernstein's abstract Plancherel formula (\cite{BernPL}) for the space $\sm_c(\GL_{n-1}(F)\backslash \GL_{n-1}(E))$. In \cite[Section 3]{HO}, a strategy was then proposed, using the classification of distinguished discrete series in terms of cuspidal representations, to reduce the RLC(n,n-1) (and actually RLC(n,n-2)) from the case of discrete series to Ok's cuspidal case. Namely \cite[Section 3]{HO} proves that RC(n,n-1) holds for discrete series of $\GL_n(E)$ if one moreover assumes that the discrete series representation $\pi$ is conjugate selfdual. Following \cite{HO}, we proved with Offen in \cite{MO} that one can indeed reduce to the conjugate selfdual case for square integrable representations, except in the case where $\pi$ is a generalized Steinberg representation of the form $\St_2(\rho)$, where $\rho$ is a cuspidal representation of some $\GL_r(E)$ for $r\geq 2$. In other words we obtained RLC(n,n-1) as well as RLC(n,n-2), except for the "exceptional" discrete series representations of the form $\St_2(\rho)$. These specific discrete series representations have actually nothing exceptional, and the incompleteness of the result in \cite{MO} is due to the fact that its method relies on a trick coming from explicit properties of $L$ functions and gamma factors, which does not work in this special case if one restricts to twists by generic distinguished representations of $\GL_{n-1}(E)$. In fact the proof of the RLC(n,n) is complete by \cite{MO}. 

Here, either when $[E:F]=2$ or $E=F$, we directly extend Ok's proof, using Bernstein's abstract Plancherel formula for the Harish-Chandra Schwartz space 
$\CC(H_{an-1}\backslash G_{an-1})$. This extension is far from being immediate and poses many technical difficulties, as well as minor extra complications when $E=F$.  

Let us denote by $Z_n$ the center of $G_n$, and by $U_n$ the unipotent radical of its standard parabolic subgroup of type $(n-1,1)$. Fix a non trivial character $\psi:E\to \BC^\times$ trivial on $F^\times$ when $[E;F]=2$, and denote by the same letter the non degenerate character of the group of unipotent upper triangular matrices of $G_n$ obtained from $\psi$ in the usual way. The basic idea is to express Whittaker functions in the Whittaker model of a square integrable representation $\pi$ of $G_n$ as Fourier coefficients of matrix coefficients thanks to \cite{LM}. Then the technical heart of this paper is to prove that if $f$ is a matrix coefficient of $\pi$, the function 
\[P_f:g \in G_{n-1} \to \int_{Z_{n-1}^{\theta}} \int_{U_n} f(u\diag(z g,1))\psi^{-1}(u)dudz\] makes sense, and actually belongs to the Harish-Chandra Schwartz space $\CC(Z_{n-1}^{\theta}\backslash G_{n-1})$. Though quite delicate, this result is facilitated by the fact that the inner integral over $U_n$ actually stabilizes. The Whittaker average of $P_f$ is then of the form $\int_{Z_{n-1}^\theta} W(\diag(zg,1))dz$ for $W\in \CW(\pi,\psi)$, and any function of the form $\int_{Z_{n-1}^\theta} W(\diag(zg,1))dz$ can be obtained by this procedure. Once this is proved, we need to verify the absolute convergence of a certain integral considered in Section \ref{sec abs cv}, to guarantee that Ok's argument extends to discrete series. We use that the the Plancherel measure of $L^2(H_{n-1}\backslash G_{n-1})$ is supported on tempered representations, as follows from the papers \cite{GO}, \cite{CZlocal} and \cite{BPGalsq}.

Though our result is optimal in terms of the class of $\pi$ to which it applies, one can imagine 
that RLC(n,n-2) holds with a similar proof (see \cite{HO} and \cite{MO}), and more optimistically its $(n,\lfloor n/2 \rfloor)$ version (see \cite{Nien} for finite fields). 

The paper is organized as follows. In Section \ref{sec prel} we introduce the material needed on reductive groups, their symmetric pairs, their Harish-Chandra Schcwartz spaces and their representations. In Section \ref{sec BFIF} we recall the abstract Fourier inversion formula for the Harish-Chandra Schwartz space of a Gelfand symmetric pair, following Bernstein. In Section \ref{sec trivial root} we verify that the central gamma and epsilon values of interest to us are indeed equal to one. Note that we had already done it in \cite{MO} when $[E:F]=2$ using Ok's result for cuspidal representations, and here when $E=F$, we again reduce the result to the cuspidal case proved in \cite{PR99} passing to the Langlands parameter's side. Section \ref{sec part Four} is where we prove the main technical result on the map $P_f$ explained above in the introduction. We then  prove our main result Theorem \ref{thm main} in the final section \ref{sec RLC}. 

Finally, we comment on the restriction to $p$-adic fields in this paper. Indeed Ok's thesis is written for any non-Archimedean local field of characteristic different from two. We believe that all the results in the present paper hold in this situation, but the majority of the references used in our paper have the characteristic zero restriction, though it is not a serious one in any of them. However justifying this latter claim carefully would either require a separate note, or make the present paper lengthy and unpleasant to read if it was to be done here. 

\subsection*{Acknowledgement.} We are grateful to Alberto Minguez for inviting us to Vienna University in June 2024, where the elaboration of this project started. We thank Jeffrey Hakim for letting us know about Ok's work and providing it to us many years ago. We thank U.K. Anandavardhanan, Dipendra Prasad, and Rapha\"el Beuzart-Plessis for useful communication and discussions on the topic of this paper. We also thank a referee for useful comments concerning a previous version of this paper, concerned with the Galois model only.

\subsection*{Update.} It was observed by Beuzart-Plessis that the main result of this paper extends to tempered representations, following the same strategy. We refer to Remark \ref{rem final 1} for more details. 

\section{Preliminaries}\label{sec prel}

\subsection{Locally profinite groups}

The letter $G$ always denotes a group, and $Z$ or $Z(G)$ its center. We write $K\leq G$ for ''$K$ is a subgroup of $G$''. We set 
\[G(k):=\{x^k,\ x\in G\}.\] If $G$ is locally compact and totally disconnected, we denote by $\delta_G$ its modulus character, and by $\sm(G)$ the space of complex valued locally constant functions on $G$. We denote by $R$ and $L$ the actions of $G$ on $\sm(G)$, given respectively by right and left translation: 
\[R(g)f(x)=f(xg)\] and \[L(g)f(x)=f(g^{-1}x).\] If $\CL$ is a subspace of $\sm(G)$, we set 
\[\CL^K=\{f\in \CL, \ \forall k \in K, \ R(k)f=f\}\] and 
\[\CL_K=\{f\in \CL, \ \forall (k,k')\in K^2, \ R(k)L(k')f=f\}.\] 
For $H\leq G$ a closed subgroup, we define $\sm(H\backslash G)$ to be the set of locally constant functions on $H\backslash G$, and denote by $\sm_c(H\backslash G)$  its subspace consisting of functions with compact support in $H\backslash G$. When $H\backslash G$ has a right $G$-invariant measure, we denote by $L^2(H\backslash G)$ the corresponding $L^2$ space. Throughout this paper we do not insist on the choice of invariant measures on homogeneous space, but they are made in a coherent enough way for all usual identities that we use to hold. 

\subsection{Non Archimedean reductive groups}

Let $F$ be a non Archimedean local field with normalized absolute value $|\ |_F$. Suppose that $G$ is (the group of $F$-points of) a connected reductive group defined over $F$. We denote by $G^1$ the intersection of the kernels of all unramified characters of $G$, and refer to \cite[p. 239]{WPL} for a less tautological definition. For example, when $G=\GL_n(F)$ we have 
\[\GL_n(F)^1=\{g\in \GL_n(F), \ |\det(g)|_F=1\}.\] 

We fix $K_0$ a maximal compact (open) subgroup of $G$ with the properties described in \cite[p.240]{WPL}. We fix \[\sigma:G\to \BC\] the logarithm of a norm function as defined in \cite[before I.2]{WPL}. We call $\sigma$ a log-norm function on $G$. It is actually a non negative map in 
$\sm(G)_{K_0}$, which is invariant under $g\to g^{-1}$. For $d\in \BR$, we then put 
\[N_d(g)=N_d^G(g):=(1+\sigma(g))^d.\] We define 
\[\sigma_*:Z\backslash G\to \BC\] by 
\[\sigma_*(g)=\inf_{z\in Z} \sigma(zg).\] It is a log-norm on $Z\backslash G$ by \cite[p. 1530]{CZlocal}, and for $d\in \BR$, we set 
\[N_{d,*}(g)=(1+\sigma_*(g))^d.\]
We also need the spherical coefficient 
\[\Xi:=\Xi_G:G\to \BC\] defined in \cite[II.1]{WPL}. Note that $\Xi=\Xi_G=\Xi_{Z\backslash G}$. It is a positive function in $\sm(G)_{K_0}$ and it is also invariant under $g\to g^{-1}$. 

Let $X$ be a set and $Y$ a subset of $X$. If $f$ is a map on $X$, we denote by $\Res_Y(f)$ its restriction to $Y$. If $f_1, \ f_2$ are maps from $X$ to $\BR_{\geq 0}$, we write \[f_1\prec f_2\] if there exists a $c\in \BR_{>0}$ such that 
$f_1\leq c f_2$. 
By definition, the Harish-Chandra Schwartz space $\CC(G)$ of $G$ is defined as 
\begin{equation} \label{eq HCS} \CC(G)=\bigcup_K \{f\in \sm(G)_K,\ \forall \ d\in \BR, \ f\prec \Xi N_{-d}\},\end{equation} where the union is over all compact open subgroups $K$ of $G$. Note that we could replace $\BR$ by any of its subsets containing an unbounded subset of $\BR_{\geq 0}$ in the above definition. 
We define the Harish-Chandra Schwartz space $\CC(G^1)$ by replacing $G$ by $G^1$ in the Equation \eqref{eq HCS}. 

\subsection{Some properties of symmetric spaces}\label{sec sym}

In this section $F$ is a non Archimedean local field of characteristic zero, $G$ is an $F$-reductive group, and $\theta$ is an $F$-rational involution of $G$. We let $H$ be a symmetric subgroup of $G$ contained in $G^{\theta}$, as in \cite[p.1531]{CZlocal}. Following \cite[Section 2]{DH} we define the following functions on $H\backslash G$:
\[\sigma^{H\backslash G}(Hg)=\sigma(\theta(g^{-1})g)\] (defined in \cite[(2.19)]{DH}) and 
\[\Xi^{H\backslash G}(Hg)=\Xi(\theta(g^{-1})g),\] (which is $\Theta_G^2$ in \cite[(2.21)]{DH}). We then set 
\[N_d^{H\backslash G}(Hg)=(1+\sigma^{H\backslash G}(Hg))^d. \]

We recall that a torus $A$ of $G$ is called $\theta$-split if it is $F$-split, and satsifies that $\theta(a)=a^{-1}$ for all $a\in A$. A parabolic subgroup $P$ of $G$ is $\theta$-split if $\theta(P)=P^-$, where $P^-$ is the parabolic subgroup of $G$ opposite to $P$. For $P$ a minimal $\theta$-split parabolic subgroup of $G$, we denote by $A_{P,\theta}$ the maximal $\theta$-split torus in the center of Levi subgroup $M:P\cap \theta(P)$. If moreover $\Delta_{P,\theta}$ is the set of simple roots (see \cite[Proposition 5.9]{HW}) of $A_{P,\theta}$ acting on the Lie algebra of $P$, we set 
\[A_{P,\theta}^+=\{a\in  A_{P,\theta}, \ \forall \alpha\in \Delta_{P,\theta}, \ |\alpha(a)|_F\leq 1 \}.\]
Then by \cite[Theorem 1.1]{BO}, there exists a finite set $\CP$ of minimal $\theta$-split parabolic subgroups of $G$, and a compact subset $\Omega$ of $G$, such that \[G=\bigcup_{P\in \CP} HA_{P,\theta}^+\Omega.\] This is called the Cartan decomposition of $H\backslash G$.

\subsection{Specific notations for symmetric pairs attached to GLn}

Here $F$ is a non Archimedean local field. We denote by $E$ a separable field extension of $F$ of degree at most $2$. For $n\geq 1$, we set \[G_n=\Res_{E/F}\GL_{n}(F)=\GL_{n}(E).\] We denote by $\theta$ the involution of $G_n$ induced by the Galois conjugation of $E/F$ when $[E:F]=2$, whereas 
$\theta:=\Ad(\diag(1,-1,\dots,(-1)^{n-1}))$ is the inner involution of $G_n$ by the matrix $\diag(1,-1,\dots,(-1)^{n-1})$ when $E=F$. Seeing $G_n$ as a Weil restriction of scalars allows us to consider 
$Z_n^{\theta}\backslash G_n$ as an $F$-reductive group with anisotropic center, this center being trivial when $E=F$. We denote by $N_{n}$ the subgroup of $G_n$ consisting of upper triangular unipotent matrices, and by $T_{n}$ its diagonal torus. The group $T_{n}$ can be parametrized by $(E^\times)^{n}$ via simple roots:
\[t:(z_1,\dots,z_{n})\to \diag(z_1\dots z_n,z_1\dots z_{n-1},\dots, z_{n})\] is a group isomorphism between $(E^\times)^{n}$ and 
$T_n$. We put $B_{n}=N_{n}T_{n}$, and denote by $B_{n}^-$ its image under transpose. If 
$\psi:E\to \BC^\times$ is a non trivial character, it defines the non degenerate character $\psi:N_{n}\to \BC^\times$ given by \[\psi(x)=\psi(\sum_{i=1}^{n-1} x_{i,i+1}).\] We set 
\[K_{n}=\GL_{n}(O_E),\] where $O_E$ is the ring of integers of $E$. For $g\in \GL_n(E)$, we put 
\[\nu_E(g)=|\det(g)|_E.\] Whether $E=F$ or not, we set \[|\ |:=|\ |_F\] and \[\nu(h)=|\det(g)|\] for $g\in \GL_n(F)$. 

When $E=F$, we have an explicit isomorphism from \[h_n:G_{\lfloor \frac{n+1}{2} \rfloor }\times G_{\lfloor \frac{n}{2} \rfloor}\simeq H_n\] defined in \cite[Section 2.1]{MatCrelle}. We can define the following character of $H_n$:

\[ \delta_n(h(g_1,g_2)):=\nu(g_1)\nu(g_2)^{-1}.\]

We denote by $P_{a_1,\dots,a_r}=M_{a_1,\dots,a_r}N_{a_1,\dots,a_r}$ the standard Levi decomposition of the standard parabolic subgroup of $G_n$ attached to the composition $(a_1,\dots,a_r)$ of $n$. We denote by $P_{n}$ the mirabolic subgroup of $G_n$, which consists of matrices in $P_{n-1,1}$ with bottom row equal to \[\eta_n:=(0\ \dots \ 0 \ 1),\] and we set $U_{n}:=N_{n-1,1}$. For $t\in E^\times$, we set 
\[z_{n-1}(t)=\diag(tI_{n-1},1)\] and for $x\in E^{n-1}$, we set 
\[u(x)=\begin{pmatrix} I_{n-1} & x \\ 0 & 1\end{pmatrix}\in U_{n}.\]

When $E=F$, we denote by $S_{2n}$ the Shalika subgroup of $G_{2n}$: it is the set of matrices of the form $s(g,x)=\diag(g,g)\begin{pmatrix} I_n & x \\ & I_n \end{pmatrix}$ when $(g,x)$ varies in $\GL_n(F)\times \CM_n(F)$. 

\subsection{The Harish-Chandra Schwartz space of a symmetric space}\label{sec HC sym}

Here $F, G$ and $H$ are as in Section \ref{sec sym}. The Harish-Chandra Schwartz space $\CC(H\backslash G)$ is defined in \cite[Section 4]{DH} as
\[\CC(H\backslash G)=\bigcup_K \{f\in \sm(H\backslash G)^K,\ \forall \ d\in \BR, \ f\prec \Xi^{H\backslash G} N_{-d}^{H\backslash G}\},\] where the union is again over all compact open subgroups $K$ of $G$.  Again we can replace $\BR$ by any unbounded subset of $\BR$ containing $\BR_{\geq 0}$ in its definition. The space $\CC(H\backslash G)$ is equipped with an LF-space structure by its very definition due to the following fact: the topology on 
$\CC(H\backslash G)_K$ is given by the family of semi-norms 
\[\nu_d(f)=\sup_{H\backslash G} f\times (\Xi^{H\backslash G})^{-1} \times N_{d}^{H\backslash G}\] for $d\in \BN$, which endow it with a Fr\'echet space structure. We observe that \[\CC(H\backslash G)\subseteq L^2(H\backslash G)\] thanks to \cite[Lemma 2.1]{DH}. On the other hand 
\[\sm_c(H\backslash G)\subseteq \CC(H\backslash G).\]  It is known that both inclusions are dense with respect to the topology of the bigger space. 

Because later we will use the results of \cite{BernPL}, we note that it is explained in \cite[Section 2.6.1]{DHS} that the definition of the Harish-Chandra Schwartz space of $H\backslash G$ that we use here coincides with that given in \cite{BernPL}.

We mention that the space $H\backslash G$ has polynomial growth in the sense of Bernstein as shown in \cite[Example 4.3.6]{BernPL}. We note that Bernstein assumes the two properties (i) and (ii) in \cite[Example 4.3.6]{BernPL} to claim polynomial growth, but they have since been verified thanks to \cite[Proposition 5.9]{HW} and the Cartan decomposition of \cite[Theorem 1.1]{BO}. This polynomial growth property is required for the Bernstein's inversion formula recalled in Section \ref{sec BFIF} of the present paper. 

\subsection{Representations} 

Let $F$ be a non Archimedean local field of characteristic zero, and $G$ be an $F$-reductive group. We consider smooth complex representations of $G$ and its closed subgroups, and call them representations. If $(\pi,V_\pi)$ is a representations of $G$ and $K\leq G$, we denote by 
$V_\pi^K$ the space of $K$-invariant vectors in $V_\pi$. The different types of inductions from a closed subgroup $K$ to $G$ are always normalized, and we denote by $\ind_K^G(\tau)$ the representation of $G$ compactly induced from the representation $\tau$ of $K$, whereas we use the notation $\Ind_K^G(\tau)$ for the full induced representation. We denote by $\Irr(G)$ the class of irreducible representations of $G$, by $\uni(G)$ its subclass consisting of unitary representations, by $\temp(G)$ the subclass of $\uni(G)$ consisting of tempered representations, and by $\usq(G)$ the subclass of $\temp(G)$ consisting of square-integrable representations. When $\pi\in \uni(G)$, there exists up to homothety a unique $G$-invariant scalar product on $\pi$ which we denote $\langle \ , \ \rangle_\pi$, and we denote by $\Pi$ the Hilbert completion of $\pi$. If $H$ is a subgroup of $G$, we say that 
a representation $\pi\in \Irr(G)$ is $H$-distinguished if $\Hom_H(V_\pi,\BC)$ is not reduced to zero. We say that the pair $(G,H)$ is a Gelfand pair if $\Hom_H(V_\pi,\BC)$ has dimension at most one for all $\pi\in \Irr(G)$. If $H$ is a symmetric subgroup of $G$, we say that $(G,H)$ is a symmetric Gelfand pair if it is a Gelfand pair. We shall later specialize to general linear groups, when considering generic representations. We fix $\psi:E\to \BC^\times$ a non trivial character of $E$. We say that $\pi\in \Irr(G_n)$ is generic if $\Hom_{N_n}(V_\pi,\psi)$ is not reduced to zero, and this notion is independent of the non trivial character $\psi:E\to \BC^\times$. It is known since \cite{GK} that $\Hom_{N_n}(V_\pi,\psi)$ can have dimension at most one, and when this dimension is one, we denote by $\CW(\pi,\psi)$ the Whittaker model of $\pi$ with respect to $\psi$. We denote by $\gen(G_n)$ the class of generic representations of $G_n$. Moreover when $E=F$ and $n$ is even, we define the Shalika character $\Theta:S_n\to \BC^\times$ by the formula 
\[\Theta(s(g,x))=\psi(\Tr(x)).\] We say that $\pi\in \Irr(G_n)$ admits a Shalika model if $\Hom_{S_n}(\pi,\Theta)\neq 0$. We denote by $\shal(G_n)$ the set of irreducible representations with a Shalika model. By \cite{Fli} and \cite{JR}, it is known that $\Hom_{H_n}(\pi,\BC)$ is one dimensional whenever $\pi\in \dist(G_n)$, and when $n$ is even it also follows from 
\cite{JR} that $\Hom_{S_n}(\pi,\Theta)$ has dimension one for all $\pi \in \shal(G_n)$. We will use the notational rule 
\[\Irr{\star}(G)\cap \Irr{\bullet}(G)=\Irr{\star,\bullet}(G).\] 
When $E=F$, it is known that 
\[\gdist(G_{2n})=\gshal(G_{2n})\] thanks for example to \cite{MatCrelle} and \cite{MatBLMS}, hence in particular 
\[\tdist(G_{2n})=\tshal(G_{2n}).\]

The above equalities hold for Archimedean $F$ as well due to the \cite[Appendix B]{Sign}, \cite[Theorem 2.1]{CJLT}, and the fact that for 
$\GL_2(\BR)$ and $\GL_2(\BC)$, discrete series representations admit a Shalika model if and only if they admit a linear period if and only if their central character is trivial as can be seen by a standard argument using Rankin-Selberg Zeta integrals. There is also an argument of Gan using the Theta correspondence which also works uniformly over local fields (\cite{Ganthetaperiods}). Actually the following result due to Friedberg and Jacquet shall be enough for our purpose. 

\begin{lem}\label{lm FJ}
For $E=F$ and $n\geq 1$, one has $\shal(G_{2n})\subseteq \dist(G_{2n})$. 
\end{lem}
\begin{proof}
This follows from the argument of \cite[p.67]{JR} and \cite{FJ} which is for all local fields of characteristic zero. 
\end{proof}

 For $\pi$ a representation of $G$, we say that $f:G\to \BC$ is a matrix coefficient of $\pi$ if there are $v\in V_{\pi}$ and $v^\vee$ in the space $V_{{\pi}^\vee}$ of the contragredient representation of $\pi^\vee$ of $\pi$, such that for all $g\in G$ we have 
\[c(g)=\langle \pi(g)v, v^\vee\rangle.\] We denote by $\CA_2(G)$ the subspace of $\sm(G)$ generated by matrix coefficients of (irreducible) square integrable representations of $G$. We recall the following inclusion from \cite[Corollaire III.1.2]{WPL}.

\begin{prop}\label{prop WHC}
Let $G$ be a reductive group defined over $F$, then \[\Res_{G^1}(\CA_2(G))\subseteq \CC(G^1).\] 
\end{prop}

In view of \cite[Lemme II.4.4]{WPL}, we obtain the following consequence.

\begin{cor}\label{cor res schw}
Let $f\in \CA_2(G_n)$, then for any $1\leq k \leq n$, the map 
\[g\in \GL_k(E)^1\to f(\diag(g,I_{n-k}))\] belongs to $\CC(\GL_k(E)^1)$. 
\end{cor}
 
We recall from \cite[III.17]{WPL} (see also \cite[Section 2.2, discussion after (2.2.6)]{BPlocal} for more details on this claim) that if $\pi\in \temp(G)$, the natural action of 
$\sm_c(G)$ on $\pi$ extends to $\CC(G)$, and for $f\in \CC(G)$ and $v\in V_\pi$, the vector 
$\pi(f)v$ is characterized by the fact that $\langle \pi(f)v,v^\vee \rangle=\int_G f(g)\langle \pi(g) v, v^\vee \rangle $ for all $v^\vee\in V_\pi^\vee$, where the integrals are absolutely convergent. Let $f\in  \CC(G)$, this implies, by considering a compact open subgroup $K$ such that 
$f\in \CC(G)_K$, that \[\lambda(\pi(f)v)=\int_G f(g)\lambda(\pi(g)v)dg\] for all $v\in V_\pi$ and all $\lambda$ in the algebraic dual of $V_\pi$. Further, by convoluting $f$ by an apporpriate function of $\sm_c(G)$ on the left which is permissible according to \cite[Lemme III.6.1]{WPL}, we obtain for all $x\in G$: \begin{equation}\label{eq sakz}\lambda(\pi(x^{-1})\pi(f)v)=\int_G f(xg)\lambda(\pi(g)v)dg.\end{equation} We shall use this fact.
 
When $P=MU$ is a parabolic subgroup of an $F$-reductive group $G$ with Levi subgroup $M$ and unipotent radical $U$, we denote by $A_M$ the central split component of $M$. For $\pi$ a finite length representation of $G$, we denote by $(V_\pi)_P$ its normalized Jacquet module with respect to $P$, and by $V_\pi(U)$ the kernel of the projection $V_\pi\to (V_\pi)_P$. We denote by \[\CE(A_M ,(V_\pi)_P)\] the set of central exponents of $A_M$ in $(V_{\pi})_P$, i.e. the central characters of the irreducible subquotients of $(V_\pi)_P$ restricted to $A_M$. 

For $\chi:E^\times\to \BC$ a character, we denote by $\Re(\chi)$ the unique real number $r$ such that $|\chi(x)|=|z|_E^r$ for all $z\in E^\times$, and call it the real part of $\chi$. We say that $\chi$ is positive if $r>0$. If $\pi\in \usq(G_n)$, it follows from Casselman's criterion \cite[Theorem 4.4.6]{Cas} as explicated in \cite[before Theorem 3.2]{MatRT} that for $1\leq k\leq n$ and $\chi\in \CE(Z_{M_{k,n-k}},(V_\pi)_{P_{k,n-k}})$, the character \[z_k\in E^\times \to \chi(\diag(z_kI_k,I_{n-k}))\] is positive.

\subsection{Tempered symmetric pairs}\label{sec TP}

Here $F, G$ and $H$ are as in Section \ref{sec sym} again. In \cite[Definition 1.3]{CZlocal}, Zhang defines the notion of \textit{very strongly discrete} symmetric pair $(G,H)$, by the requirement that the integral \[\int_{Z_H\backslash H} f(h)dh\] is absolutely convergent for all $f\in \CC(Z\backslash G)$. By \cite[Section 3.2]{CZlocal}, the pair $(G,H)$ is very strongly discrete if and only if there exists a natural integer $d$ such that  
\[\int_{Z\cap H\backslash H}\Xi(h)N_{-d,*}(h)dh<+\infty.\] 
Now by \cite[Lemma 2.7.1]{BPGalsq}, the above characterization of very strongly discrete symmetric pairs shows that the very strongly discrete symmetric pairs of Zhang are the same as the \textit{tempered symmetric pairs} of Beuzart-Plessis as defined in \cite[Section 2.7]{BPGalsq}. We decide to opt for the terminology tempered symmetric pair. Indeed, the following implication is proved in 
\cite[Proposition 2.7.1]{BPGalsq} (its converse has also been proven already but not yet published by Beuzart-Plessis).

\begin{prop} 
If $(G,H)$ is a tempered symmetric pair, then the class of Plancherel measures of the unitary representation $L^2(H\backslash G)$ of $G$ is supported on $\temp(G)$. 
\end{prop}

 It is proved in \cite[Section 3.2.1]{CZlocal} that whenever $E/F$ is a quadratic extension, $H$ is an $F$-reductive group, and $G$ is the $F$-points of the Weil restriction of scalars from $E$ to $F$ of $H$, then the Galois symmetric pair $(G,H)$ is tempered. Together with \cite[Remark 3.4]{CZlocal} and \cite[Corollary 5.16]{GO}, this implies that the symmetric pairs $(G_n,H_n)$ of interest to us are always tempered. 

We now assume that the symmetric pair $(G,H)$ is tempered, and we set $Z_H:=Z\cap H$. We moreover assume that $Z_H\backslash Z$ is compact. In this situation the integral $\int_{Z_H\backslash H} f(h)dh$ is a absolutely convergent for all $f\in \CC(Z_H\backslash G)$. In particular we have a projection map 
\[p_H: \CC(Z_H\backslash G)\to \sm(H\backslash G)\] defined by 
\[p_H(g)=\int_{Z_H\backslash H} f(hg)dh.\] We observe that $(Z_H\backslash G,Z_H\backslash H)$ is also a tempered symmetric pair, and that the spaces $\CC(H\backslash G)$ and $\CC(\frac{H}{Z_H}\backslash \frac{G}{Z_H})$ tautologically identify. The following fact follows from \cite[Lemma 4.3]{CZlocal} applied to the tempered symmetric pair $(Z_H\backslash G,Z_H\backslash H)$.  

\begin{prop}\label{prop CZ}
Let $(G,H)$ be a tempered symmetric pair such that $Z_H\backslash Z$ is compact, then the projection $p_H$ sends the Harish-Chandra Schwartz space $\CC(Z_H\backslash G)$ to $\CC(H\backslash G)=\CC(\frac{H}{Z_H}\backslash \frac{G}{Z_H})$.
\end{prop}

\begin{rem}
The above proposition applies to $(G_n, H_n)$. 
\end{rem}

\subsection{The Whittaker Harish-Chandra Schwartz space}

Here $F$ is a non Archimedean local field and $G$ an $F$-reductive group. We fix $U_0$ to be the unipotent radical of a minimal parabolic subgroup $P_0$ of $G$. We fix a Levi subgroup $M_0$ of $P_0$. We denote by $\d_0$ the modulus character of $P_0$. We denote by $\mathfrak{a}_0$ the tensor product over $\BZ$ of $\BR$ with the lattice of algebraic characters of $M_0$. We denote by $H_0:M_0\to \mathfrak{a}_0$ the function 
defined in \cite[p. 240]{WPL}. Finally we fix $\psi:U_0\to \BC^\times$ a non-degenerate character as in \cite[Section 3.1]{DWhit}. In \cite{DWhit}, Delorme defines the Whittaker Harish-Chandra Schwartz space $\CC(U_0\backslash G,\psi)$ as the subspace of the smooth induced representation 
$\Ind_{U_0}^G(\psi)$ of functions satisfying 
\[f\prec \delta_0^{1/2}(1+H_0)^{-d}\] for any positive integer $d$. The following lemma is an immediate consequence of \cite[Proposition II.4.5]{WPL}, \cite[Inequality (6) p.242, before Section I.2]{WPL} and the Iwasawa decomposition.

\begin{lem}\label{lm proj to WHCS space}
For $f\in \CC(G)$, the map $\int_{U_0} f(u)\psi^{-1}(u)du$ is absolutely convergent. Moreover the map 
\[W_f:g\mapsto \int_{U_0} f(ug)\psi^{-1}(u)du\] belongs to $\CC(U_0\backslash G,\psi)$. 
\end{lem}

\section{The abstract Fourier inversion formula for symmetric spaces}\label{sec BFIF}

Here $F$ is a non Archimedean local field of characteristic different from two. We state a consequence of the main result of \cite{BernPL} for Gelfand symmetric pairs. 

\begin{thm}\label{thm Bern}
Let $(G,H)$ be a Gelfand symmetric pair and let $\mu$ belong to the class of Plancherel measures of the right regular representation $L^2(H\backslash G)$. Then $\mu$ is supported on a certain subspace $\Htemp(G)$ of $\udist(G)$. It has the property that for each representation $\pi\in \Htemp(G)$, there is a generator 
$\lambda_{\pi}$ of $\Hom_H(\pi,\BC)$ such that:
\begin{enumerate}
\item if $f\in \CC(H\backslash G)$ and $v\in \pi$, the integral \[\langle f , v\rangle_{\lambda_{\pi}}:=\int_{H\backslash G} f(g)\overline{\lambda_{\pi}(\pi(g)v)}dg\] is absolutely convergent,
\item if $K$ is a compact open subgroup of $G$ fixing $f$ on the right, then for any choice of orthonormal basis $B_{\pi}^K$ of $V_\pi^K$, one has: 
\[f(He)=\int_{\Htemp(G)} \sum_{v\in B_{\pi}^K}\langle f , v\rangle_{\lambda_{\pi}}\lambda_{\pi}(v) d\mu(\pi).\] 
\end{enumerate}
\end{thm}
\begin{proof}
According to \cite{BernPL}, and as further detailed in \cite[Section 3.3]{WWL}, there exists a pointwise defined (see  \cite[Definition 3.1]{WWL}) map 
\[\alpha:\sm_c(H\backslash G)\to \int_{\uni(G)}^\oplus \Pi d\mu(\pi)\] with some extra properties that we now explicate. 
First, we can suppose that for each $\pi$, the map $\alpha_\pi$ sends $\sm_c(H\backslash G)$ to the smooth part $\pi$ of $\Pi$. Denote by \[\beta_{\pi}: \pi \to \sm(H\backslash G)\] the adjoint of $\alpha_{\pi}$ defined by the relation 
\begin{equation}\label{eq adjoint} \langle \alpha_\pi(f), v\rangle_\pi=\int_{H\backslash G} f(g) \overline{\beta_\pi(v)(g)} dg.\end{equation} By Frobenius reciporicity, there exists $\lambda_{\pi}\in \Hom_H(\pi,\BC)$ such that for any $v\in \pi$, and any $g\in G$, one has 
\begin{equation}\label{eq beta} \beta_{\pi}(v)(Hg)=\lambda_{\pi}(\pi(g)v),\end{equation} as explained in \cite[Section 4.2]{HO}. In particular, the 
measure $\mu$ is actually supported on $\udist(G)$. Then, because $H\backslash G$ has polynomial growth, Bernstein asserts on \cite[p. 668]{BernPL} that $\mu$ is in fact supported on the set of $H$-tempered representations, i.e. the set $\Htemp(G)$ of representations $\pi$ such that each $\alpha_\pi:\sm_c(H\backslash G)\to \pi\subseteq \Pi$ extends (necessarily uniquely) as a continuous map which we still denote $\alpha_\pi:\CC(H\backslash G)\to \Pi$. Note that the image of this extension again lands inside $\pi$ by smoothness of the functions in the Harish-Chandra Schwartz space. For $H$-tempered $\pi$, it is then explained in \cite[p. 689]{BernPL} that for any $v\in \Pi$, the map $\beta_{\pi}(v)$ has moderate growth in the sense that $N_d^{H\backslash G}\beta_\pi(v)$ belongs to $L^2(H\backslash G)$ for some $d>0$. We denote by $\CC^w(H\backslash G)$ the space of functions of moderate growth in $\sm(H\backslash G)$, hence 
\[\beta_\pi:\pi\to  \CC^w(H\backslash G).\] It follows that if $\pi$ is $H$-tempered, the integral $\int_{H\backslash G} f(g)\overline{\beta_\pi(v)(g)} dg$ is absolutely convergent whenever $f\in \CC(H\backslash G)$. Approaching $f\in \CC(H\backslash G)$ by functions of the form $f_k=\mathbf{1}_{C_k}f$, for $\mathbf{1}_{C_k}$ the characteristic function of $C_k$ in an increasing sequence of compact open subsets of $H\backslash G$ exhausting it, we see that Equality \eqref{eq adjoint} still holds for $f$ by continuity of $\alpha_\pi$ and the dominated convergence theorem. Note that 
$\beta_{\pi}$ can be thought of as a linear map from $\pi$ to the strong Hermitian dual $\overline{\CC(H\backslash G)'}$ of $\CC(H\backslash G)$. Indeed  
$\CC^w(H\backslash G)$ identifies with a subspace of $\overline{\CC(H\backslash G)'}$ by setting 
\[ \langle f^w ,f \rangle =\int_{H\backslash G} f^w(x) \overline{f}(x) dx \] for $f\in \CC(H\backslash G)$ and $f^w\in \CC^w(H\backslash G)$. The main result of \cite{BernPL} then states that for any $f\in \CC(H\backslash G)\subseteq \CC^w(H\backslash G) \subseteq \overline{\CC(H\backslash G)'}$, the following equality holds in $\overline{\CC(H\backslash G)'}$:
\[ f=\int_{\Htemp(G)} \beta_{\pi}(\alpha_\pi(f))d\mu(\pi).\] Implicit in this equality is the existence of its right hand side in the sense of Gelfand-Pettis integrals (see \cite[p.77]{Rud}). 

Note that for any $\phi\in  \CC(H\backslash G)$, the map $L\to L(\phi)$ is obviously continuous on the strong dual $\CC(H\backslash G)'$. In particular this implies that for any 
$\phi$ in $\CC(H\backslash G)$ one has 
\begin{equation}\label{eq **} \langle f, \phi \rangle =\int_{\udist(G)} \langle \beta_{\pi}(\alpha_\pi(f)), \phi \rangle d\mu(\pi).\end{equation} Now consider $f\in \CC(H\backslash G)^K$ as in the statement, so that in particular each $\beta_{\pi}(\alpha_\pi(f))$ is also right $K$-invariant. Taking for $\phi$ an appropriate multiple of the $\triv_{H\backslash HK}$, we obtain 
\[f(He)=\int_{\udist(G)} \beta_{\pi}(\alpha_\pi(f))(He)d\mu(\pi)=\int_{\udist(G)}\lambda_{\pi}(\alpha_\pi(f))d\mu(\pi),\] where the second equality holds thaks to Equation \eqref{eq beta}.  
Finally take any orthonormal basis $B_{\pi}^K$ of the finite dimensional Hilbert space $\pi^K$, we can write 
$ \alpha_\pi(f)=\sum_{v\in B_{\pi}^K} \langle \alpha_\pi(f),v \rangle_{\pi} v$, and the statement of theorem follows from 
Equation \eqref{eq beta}. 
\end{proof}

\begin{rem}\label{supp}
We emphasize that according to Section \ref{sec TP}, whenever the pair $(G,H)$ is tempered, for example when $(G,H)=(G_n,H_n)$, the Plancherel measure above is actually supported on $\tdist(G)$. In other words we can, and we will replace the set $\Htemp(G)$ by the set 
$\Htemp(G_n)\cap \tdist(G_n)$ in the inversion formula of Theorem \ref{thm Bern}.
\end{rem}

\begin{rem}
There is another notion of $H$-temperdness defined in \cite{Tak}. We did not try to compare it with that of \cite{BernPL} but if they were to agree, then it would follow from \cite[Corollary 3.14]{Tak} that actually $\tdist(G)\subseteq \Htemp(G)$ always. We do not need to know this in the present work.
\end{rem}

\section{Gamma factors and root numbers of pairs of distinguished representations, and a global application}\label{sec trivial root}

Here $F$ is any local field of characteristic zero. Let $\pi$ and $\pi'$ be irreducible representations of $G_n$ and $G_m$ respectively, and $\psi:E\to \BC^\times$ be a non trivial character. Using the Langlands classification for irreducible representations of $G_n$ in the non Archimedean case, Jacquet, Piatetski-Shapiro and Shalika associate in \cite{JPSS} and \cite{Jarch} to the triple $(\pi,\pi',\psi)$ the following meromorphic functions in the complex variable $s$, the definition of which we do not recall: $L(s,\pi,\pi'),$ $\epsilon(s,\pi,\pi',\psi) $ and $\gamma(s,\pi,\pi',\psi).$

Now according to \cite{Hllc} and \cite{HTllc} in in the $p$-adic case, and \cite{Larch} (see also \cite{Knapp}) in the Archimedean case, to $\pi$ and $\pi'$ are attached Langlands parameters that we denote by $\phi_\pi$ and $\phi_\pi'$. They are finite dimensional representations of the Weil group $W_F$ of $F$ when $F$ is Archimeden, and of its Weil-Deligne group $W_F\times \SL_2(\BC)$ otherwise. We say that $\phi_{\pi}$ is \textit{symplectic} if it preserves a non degenerate alternating linear map, in which case we say that \textit{$\pi$ is symplectic}. It is known that one has the following equalities between the factors of Jacquet, Piatetski-Shapiro and Shalika and the factors of Artin, Deligne and Langlands, where in the Archimedean case the equality is actually the definition:

\[L(s,\pi,\pi')=L(s,\phi_\pi\otimes \phi_{\pi'}),\]
\[\e(s,\pi,\pi',\psi)=\e(s,\phi_\pi\otimes \phi_{\pi'},\psi)\] and 
\[\gamma(s,\pi,\pi',\psi)=\gamma(s,\phi_\pi\otimes \phi_{\pi'},\psi).\]

The value $\e(1/2,\pi,\pi',\psi)$ is called the root number attached to $\pi$, $\pi'$ and $\psi$. We want to claim that root numbers attached to a pair of irreducible distinguished representations are trivial (assuming that $\psi$ is trivial on $F$ when $[E:F]=2$). This is actually proved in \cite[Theorem 3.6]{MO} when $[E:F]=2$. In the case where $E=F$, we deduce it passing to the dual side of parameters, and using the following result, which follows from \cite[Theorem 3.12, Theorem 3.20]{Sign} and \cite[Corollary 3.4]{Sign} and rely on the work of many others when $F$ is $p$-adic, and \cite[Appendix B]{Sign} when $F$ is Archimedean. 
 
\begin{thm}\label{thm symplectic}
Let $n\geq 1$, and suppose that $E=F$. If $\pi\in \dist(G_{2n})$, hence in particular if $\pi\in \shal(G_{2n})$, then $\pi$ is symplectic. 
\end{thm}

Then we have the following theorem, proved in \cite[Theorem 3.6]{MO} in the Galois case, and is essentially proved in \cite[Alternate proof of Proposition 2.1]{PR99} (see also \cite{R}) otherwise. 

\begin{thm}\label{thm trivial local root}
Let $n,m\geq 1$. Let $\pi\in \dist(G_{an})$ and $\pi'\in \dist(G_{am})$, and assume that $\psi$ is trivial on $F$ when $[E:F]=2$. Then 
\[\e(1/2,\pi,\pi',\psi)=1.\] When $E=F$, if one takes $\pi\in \shal(G_{an})$ and $\pi'\in \shal(G_{am})$, the conclusion 
remains unchanged. 
\end{thm}
\begin{proof}
It is sufficient to consider the case where $E=F$. We need to prove that $\e(\phi_{\pi}\otimes \phi_{\pi'},\psi)=1$. We refer to \cite[Discussion before Section 2.4]{Sign} for generalities on root numbers used in this proof. When $F$ is Archimedean, the result is proved in \cite[Alternate proof of Proposition 2.1]{PR99}. When $F$ is $p$-adic and $\phi_{\pi}$ and $\phi_{\pi'}$ are trivial on $\SL_2(\BC)$, this is also proved there. In general, we write $\mathrm{Sp}(k)$ for the $k$-dimensional irreducible algebraic representations of $\SL_2(\BC)$. Both $\phi_{\pi}$ and $\phi_{\pi'}$ are symplectic thanks to Theorem \ref{thm symplectic}, and in particular $\e(\phi_{\pi}\otimes \phi_{\pi'},\psi)$ does not depend on $\psi$ and we write it $\e(\phi_{\pi}\otimes \phi_{\pi'})$. Moreover by symplecticity both $\phi_{\pi}$ and $\phi_{\pi'}$  decompose as a sum of representations of the follwing type: $\mathrm{Sp}(k,\phi):=\phi\otimes \mathrm{Sp}(k)$ with $\phi$ an irreducible representations of $W_F$, which is symplectic when $k$ is odd and orthogonal when $k$ is even, and representations of the form $\tau\oplus \tau^\vee$ for $\tau$ irreducible. This reduces the problem to proving two equalities. First that $\e(\tau_0,\tau_1\oplus \tau_1^\vee)=1$ when $\tau_0$ is symplectic, but this follows at once from basic properties of root numbers. Second that $\e(\mathrm{Sp}(k,\phi)\otimes \mathrm{Sp}(l,\phi'))=1$ when both the pairs $(k,\phi)$ and $(l,\phi')$ satisfy the parity conditions described above. 
Now we recall that \[\mathrm{Sp}(k,\phi)\otimes \mathrm{Sp}(l,\phi')=\mathrm{Sp}(k+l-1,\phi\otimes \phi')\oplus \dots \oplus \mathrm{Sp}(|k-l|+1,\phi\otimes \phi').\] The root number of all summands above are all signs and equal to each other by \cite[Section 2, (a)]{Sign}, hence this gives the triviality as soon as $k$ or $l$ is even. The remaining case is when both $k$ and $l$ are odd, and both $\phi$ and $\phi'$ are symplectic. In this case we claim that \[\e(\mathrm{Sp}(k+l-1,\phi\otimes \phi'))=\e(\phi\otimes \phi').\] Indeed decompose $\phi\otimes \phi'$ as a direct sum of irreducible selfdual representations of $W_F$, use additivity of root numbers and and apply the formula in \cite[Section 2, (a)]{Sign} again. The triviality now follows from \cite[Alternate proof of Proposition 2.1]{PR99}.
\end{proof}

\begin{rem}
The cuspidal case of the above result, proved in \cite[Alternate proof of Proposition 2.1]{PR99}, relies on the result of Deligne (\cite{Del}) giving the sign of the root number of orthogonal representations of $W_F$ in terms of lifting to spin groups. One could also give a direct proof \`a la Ok which would be quite lengthy. Conversely \cite[Theorem 3.6]{MO}, the cuspidal case of which is one of the main results of \cite{Ok}, can be obtained by an argument on the Galois side. Indeed in the case of Galois models, cuspidal distinguished representations are known to be characterized by the fact that their Langlands parameter is conjugate orthogonal in the terminology of \cite{GGP}, and conjugate orthogonal root numbers are known to be trivial by \cite{GGP}. Note that this triviality is again proved using the result of Deligne. 
\end{rem}

The above theorem has the following corollary which is proved in \cite[Theorem 3.7]{MO} when $[E:F]=2$. 

\begin{cor}\label{cor trivial gamma}
Let $\pi\in \tdist(G_{an})$ and $\pi'\in \ugdist(G_{am})$, and assume that $\psi$ is trivial on $F$ when $[E:F]=2$. Then 
\[\gamma(1/2,\pi,\pi',\psi)=1.\] 
\end{cor}
\begin{proof}
The proof of \cite[Theorem 3.7]{MO} adapts verbatim when $E=F$, using selfduality of irreducible distinguished representations. 
\end{proof}

\begin{rem}
It is not true that the above equality extends to larger classes of irreducible representations (actually the class where $\pi$ varies can slightly be extended, see \cite[Theorem 3.7]{MO}). This is explained in \cite[Example 3.8]{MO} when $[E:F]=2$, and the examples there are easily modified to the situation where $E=F$, essentially by erasing the letter $\sigma$. 
\end{rem}

Following \cite{PR99}, we state a corollary of Theorem \ref{thm trivial local root} when $E=F$, which was conjectured by Prasad and Ramakrishnan in \cite{PR99}. The first part of this corollary is already known thanks to the works \cite{CKPSS}, \cite{GRS}, \cite{Kim}, \cite{JiSo} and \cite{JiSo2} on local and global converse theorems and functoriality: by \cite{GRS} the symplectic representation $\pi$ as below is a weak lift (see \cite{CKPSS}) of a generic cuspidal automorphic representation representation $\sigma$ of $\SO_{n+1}(\BA)$, then by \cite{Kim} the representation $\sigma$  strongly lifts (see \cite[Definition 2.1]{Kim}) to $\GL_n(\BA)$, but then by \cite{JiSo} and \cite{JiSo2},  the Langlands parameters of the local components of $\pi$ are equal to the Langlands parameters of the local components of $\sigma$. The last part of the previous statement implies that all local components of $\pi$ are symplectic. 

\begin{cor}[of Theorem \ref{thm trivial local root}, {\cite[Conjecture I and III]{PR99}}] 
Let $\BA$ be the ring of adeles of a number field, let $n\geq 1$, and let $\pi$ be a cuspidal automorphic representations of $\GL_{2n}(\BA)$ which is symplectic, i.e. such that its partial exterior square L function has a pole at $s=1$. Then all local components of $\pi$ have a symplectic Langlands parameter. In particular if $m\geq 1$ and $\pi'$ is a cuspidal automorphic representations of $\GL_{2m}(\BA)$ which is symplectic as well, the global root number $\e(1/2,\pi,\pi')$ is trivial. 
\end{cor}
\begin{proof}
By \cite{JSext}, the representation $\pi$ has a Shalika period, hence all its local components have a Shalika model.
\end{proof}

\begin{rem}
Of course the analogue statement holds when $[E:F]=2$, with the same proof. It is \cite[Theorem 6.6]{MO}, and was conjectured in [AnaRoot, Conjecture 5.1]. 
\end{rem}

A natural further question is whether $L(1/2,\pi,\pi')$ is necessarily nonzero or not.\\

The remaining sections are devoted with the proof of the relative converse theorems. 

\section{Partial Fourier coefficient of square integrable matrix coefficients}\label{sec part Four}

From now on, we focus on the group $G_n$.  

\subsection{Stable integrals of smooth functions on some unipotent groups}\label{sec stable int}

For the beginning of this section, and only there in the paper, it will be convenient to 
have the following extra notations: 

\begin{itemize} 
\item $K:=E\times E$,
\item $\psi':K\to \BC$ the character $\psi\otimes \psi^{-1}$ of $K$ for $\psi$ as we fixed before,
\item $G'_n=\GL_n(K)=\GL_n(E)\times \GL_n(E)$,
\item $\theta'(x,y)=(y,x)$,
\item $H'_n=(G'_n)^{\theta'}$,
\item $\nu_K=\nu_E \otimes \nu_E$.  
\end{itemize}

The notion of standard parabolic subgroups of $G'_n$ and their related subgroups is obvious. Then for a uniformity of the statements, we set:

\begin{itemize}
\item Either $G_n^*:=G_n$, $\psi^*=\psi$ as fixed before, $H_n^*=H_n$, etc,
\item or $G_n^*:=G'_n$, $\psi^*=\psi'$ as fixed above, $H_n^*=H'_n$, etc.
\end{itemize}

Let us recall the following useful which is proved \cite[Propositions 2.2 and 2.5]{MatPJM} when $G_n^*=G_n$ and $[E:F]=2$ (see also \cite{Fli} and \cite{Ok} for the second part). When $G_n^*=G_n$ and $E=F$ it follows from the results in \cite[Section 4]{MatCrelle}, whereas when $G_n^*=G'_n$ the second part of the statement is proved in \cite{BernMir}, and the first part is essentially contained in this paper as well. 

\begin{lem}\label{lm BFmat}
Let $\pi^*\in \ugen(G_{n}^*)$. Then the space $\Hom_{{P_{n}^*}^\theta}(W(\pi^*,\psi^*),\BC)$ has dimension one and is spanned by the linear form on 
\[\lambda_{\pi}:W^*\to \int_{{N_{n}^*}^{\theta^*}\backslash {P_{n}^*}^{\theta^*}} W^*(p)dp,\] defined by absolutely convergent integrals. In particular 
\[\Hom_{{P_{n}^*}^{\theta^*}}(\pi,\BC)=\Hom_{H_{n}^* }(\pi,\BC)\] whenever $\pi$ is distinguished by $H_{n}^*$. 
\end{lem}
\begin{proof}
When $G_{n}^*=G'_n$ we refer to \cite[Propositions 2.2 and 2.5]{MatPJM}, the proof is exactly the same. We give the details when $G_n^*=G_n$ and $E=F$, setting $\pi:=\pi^*$. By \cite{BZ2} the representation $\pi_{|P_n}$ admits a filtration in which each subquotient is of the form $(\Phi^+)^{n-k-1}(\Psi^+(\tau))$ for $0\leq k \leq n-1$ and $\tau$ an irreducible representation of $G_k$, where we refer to \cite{BZ2} for the definition of the functors $\Phi^+$ and $\Psi^+$. Moreover the representation $(\Phi^+)^{n-1}(\Psi^+(\triv_{G_0}))$ is the socle of $\pi_{|P_n}$ and corresponds to the subspace of all Whittaker functions $W\in \CW(\pi,\psi)$ such that $W_{|P_n}$ has compact support mod $N_n$. Also  according to \cite[Criterion 7.4]{BernMir} the real part of the central character of such $\tau$ must be $>-k/2$ as $\pi$ is unitary. On the other hand, thanks to \cite[4.14]{MatCrelle} we check that
\begin{enumerate} 
\item $\Hom_{P_n^\theta}((\Phi^+)^{n-k-1}(\Psi^+(\tau)),\BC)=\Hom_{G_k^\theta}(\nu^{1/2}\tau,\triv)$ if $n$ is even and $k$ is odd.
\item $\Hom_{P_n^\theta}((\Phi^+)^{n-k-1}(\Psi^+(\tau)),\BC)=\Hom_{G_k^\theta}(\nu^{1/2}\tau,\d_k^{1/2})$ if $n$ is even and $k$ is even.
\item $\Hom_{P_n^\theta}((\Phi^+)^{n-k-1}(\Psi^+(\tau)),\BC)=\Hom_{G_k^\theta}(\nu^{1/2}\tau,\triv)$ if $n$ is odd and $k$ is even.
\item $\Hom_{P_n^\theta}((\Phi^+)^{n-k-1}(\Psi^+(\tau)),\BC)=\Hom_{G_k^\theta}(\nu^{1/2}\tau,\d_k^{-1/2})$ if $n$ is odd and $k$ is odd.
\end{enumerate} 
In all cases we deduce that if $k\geq 1$, the space $\Hom_{P_n^\theta}((\Phi^+)^{n-k-1}(\Psi^+(\tau)),\BC)$ is reduced to zero. Hence 
$\Hom_{P_n^\theta}(\pi,\BC)$ is of dimension at most that of $\Hom_{P_n^\theta}((\Phi^+)^{n-1}(\Psi^+(\triv_{G_0})),\BC)=\Hom_{G_0^\theta}(\triv_{G_0},\BC)\simeq \BC$, i.e. one. It remains to prove that for $W\in \CW(\pi,\psi)$, the integral 
\[\int_{N_n^\theta\backslash P_n^\theta} W(p)dp=\int_{N_{n-1}^\theta\backslash G_{n-1}^\theta} W(\diag(h,1))dh\] converges absolutely. Indeed if so, the linear form $\lambda_{\pi}$ is automatically nonzero since the space of restrictions to $P_n $ of functions in $\CW(\pi,\psi)$ contains $(\Phi^+)^{n-1}(\Psi^+(\triv_{G_0}))$. By the Iwasawa decomposition the convergence of the integral of $\int_{N_{n-1}^\theta\backslash G_{n-1}^\theta} W(h)dh$ for any $W\in \CW(\pi,\psi)$ is reduced to that of \[\int_{T_{n-1}^\theta} W(\diag(t,1))\delta_{B_{n-1}^\theta}(t)dt=\int_{T_{n-1}} W(\diag(t,1))\delta_{B_{n-1}^\theta}^{-1}(t)dt\] for any $W\in \CW(\pi,\psi)$. Now according to 
\cite[Theorem 2.1]{MatRT}, or rather its corrected version in \cite[Theorem 2.1]{MatDer}, the convergence of the above integral amounts to that of integrals of the form \begin{equation}\label{eq int tor} \int_{F^*} c(z_k)|z_k|_F^{\frac{k(n-k)}{2}}\delta_{B_{n-1}^\theta}^{-1}(\diag(z_k I_k,I_{n-1-k}))\phi(z_k)d^*z_k\end{equation} for any 
$1\leq k\leq n-1$ such that $\pi^{(k)}$ (see \cite{BZ2}) is nonzero, where $c$ is the central character of an irreducible quotient of $\pi^{(k)}$, and $\phi\in \sm_c(F)$. Now one checks that \[\delta_{B_{n-1}^\theta}^{-1}(\diag(z_k I_k,I_{n-1-k}))=|z_k|_F^{-\lfloor \frac{k(n-1-k)}{2} \rfloor}.\] Now, as already said, by  \cite[Criterion 7.4]{BernMir} the real part of $c$ is $>-k/2$ whereas the quantity 
$\frac{k(n-k)}{2}-\lfloor \frac{k(n-1-k)}{2} \rfloor$ is either equal to $k/2$ or $k/2+1/2$, hence the Tate integral in Equation \eqref{eq int tor} converges absolutely. 
\end{proof}

We will use results and methods from both \cite{Ok} and \cite{LM} to study stable integrals of some generalized matrix coefficients over unipotent groups. The following is a generalization to all $G_n^*$ of \cite[Lemma 7.2.1]{Ok} which is for $G_n^*=G_n$, $[E:F]=2$, and a specific exhaustive family of compact open subgroups. We skip its proof which is completely similar, and relies on abelian Fourier transform properties (what is sometimes called unfolding). We set 
$\nu_*=\nu_K$ when $G_n^*=G'_n$, and $\nu_*=\nu_E$ otherwise. We also set $\chi_n^*$ to be the trivial character of $H_n^*$ except when $G_n^*=G_n$ and $E=F$, in which case $\alpha_n^*:=\delta_n^{1/2}$ if $n$ is odd and $\alpha_n^*:=\triv$ if $n$ is even. 

\begin{lem}\label{lm unfolding 1}
Let $n\geq 2$ and $1\leq k\leq n-1$. For any fixed compact open subgroup $K_0^*$ of $G_{n}^*$, there exists a compact subgroup $U^*(K_0^*)$ of $\frac{U_{k+1}^*}{(U_{k+1}^*)^{\theta^*}}$, such that for any compact open subgroup $U^*$ of $\frac{U_{k+1}^*}{(U_{k+1}^*)^{\theta^*}}$ containing $U^*(K_0^*)$ and any $\pi^*\in \ugen(G_{n}^*)$, one has the equality of absolutely convergent integrals 
\[\int_{U^*}\int_{{N_{k}^*}^{\theta^*}\backslash {H_{k}^*}^{\theta^*}}W^*(\diag(h,I_{n-k})\diag(u,I_{n-k-1}))\nu_*(h)^{(k-n+1)/2}\alpha_{k}^*(h)\]
\[\alpha_{n-1}^*(\diag(h,I_{n-1-k})^{-1} dh \psi^{-1}(u)du\]
\[=\int_{{N_{k-1}^*}^{\theta^*}\backslash {H_{k-1}^*}^{\theta^*}}W^*(\diag(h,I_{n-k+1}))\nu_*(h)^{(k-n)/2}\alpha_{k-1}^*(h)\alpha_{n-1}^*(\diag(h,I_{n-k})^{-1} dh\] for all right $K_0^*$-invariant $W^*$ in 
$\CW(\pi^*,\psi^*)$. 
\end{lem}

\begin{rem}
In \cite[Lemma 7.2.1]{Ok}, the result is stated for a specific family of $U^*$ which can be written as product of compact open subgroups with respects to the natural coordinates of $\frac{U_{k+1}^*}{(U_{k+1}^*)^{\theta^*}}$. Actually this plays absolutely no real role in its proof which relies on abelian Fourier duality properties, the group $\frac{U_{k+1}^*}{(U_{k+1}^*)^{\theta^*}}$ being locally compact and abelian. 
\end{rem}

Now let $\pi$ be a generic representation of $G_n$, for $n\geq 2$. By Lemma \ref{lm BFmat} for $G_n^*=G'_n$, if $(W,W^\vee)\in \CW(\pi,\psi)\times \CW(\pi^\vee,\psi^{-1})$, the map 
\[f_{W,W^{\vee}}:g\to \lambda_{\pi}(R(g)W\otimes W^{\vee})\] is a matrix coefficient of $\pi$. Moreover for any compact open subgroup $K_0$ of $G_n$, the map $f_{W,W^{\vee}}$ is bi-$K_0$-invariant if and only if both $W$ and $W^{\vee}$ are right $K_0$-invariant. But then, because  the integrals in Lemma \ref{lm unfolding 1} stabilize, when $G_n^*=G'_n$ and $k=n-1$, if one applies the Plancherel formula to smooth functions with compact support in $G_n$, we obtain as in \cite[Proposition 2.11]{LM} the following corollary.

\begin{cor}\label{cor Ok lm cons}
Let $K_0$ be a compact open subgroup of $G_n$, the there exists a compact open subgroup $U(K_0)$ of $U_n$ such that such that for all $f\in \sm(G_n)_{K_0}$, we have 
\[\int_{U}f(u)\psi^{-1}(u)du=\int_{U(K_0)}f(u)\psi^{-1}(u)du\] whenever $U$ is a compact open subgroup of $U_n$ containing $U(K_0)$. 
\end{cor}

We denote by \[\int_{U_n}^*f(u)\psi^{-1}(u)du\] the common value of all $\int_{U}f(u)\psi^{-1}(u)du$ for $f$ as in the above proposition.

\subsection{Partial Fourier coefficients of smooth functions}
 
Paragraph \ref{sec stable int}, and more precisely Corollary \ref{cor Ok lm cons}, allows one to define the partial Fourier coefficient 
\[f_{\psi}(g):=\int_{U_n}^*f(u)\psi^{-1}(ug)du\] for any $f\in \sm(G_n)$, and we observe that $f_{\psi}\in \sm(G_n)$. We shall soon average $\Res_{G_{n-1}}(f_{\psi})$ on $Z_{n-1}^\theta$ for $f\in \CA_2(G_n)$. In order to understand the asymptotic properties of this average, we will use the following results, the proof of which closely follows that of \cite[Theorem 2.1]{MatRT} (see \cite{MatDer} for the corrected version). 

\begin{lem}\label{lm as whit1}
Let $f$ belong to $\sm(G_n)$, then there exists an integer $b_0$ such that for any $t\in E^\times$ with $|t|_E\geq q_E^{b_0}$ and any $p\in \diag(P_{n-1},1)$, one has \[f_\psi(z_{n-1}(t)p)=0.\]
\end{lem}
\begin{proof}
We observe that for any $t\in E^\times$, any $x\in E^{n-1}$, and any $p\in P_{n-1}$, one has 
\[f(z_{n-1}(t)\diag(p,1)u(x))=f(u(tx)z_{n-1}(t)\diag(p,1))=\psi(tx_{n-1})f(z_{n-1}(t)\diag(p,1)).\]
From this relation, the vanishing property when $|z_{n-1}|\geq q_E^{b_0}$ follows from smoothness of $f$ (see \cite[Remark 2.1]{MatRT} for the full argument).
\end{proof}

\begin{lem}\label{lm as whit2}
Let $f$ belong to $\sm(G_n)(U_n)$ (for the right action $R$ of $G_n$), then there exists an integer $a_0$ such that for any $t\in E^\times$ with $|t|_E\leq q_E^{a_0}$ and any $p\in \diag(P_{n-1},1)$, one has \[f_\psi(z_{n-1}(t)p)=0.\]
\end{lem}
\begin{proof}
See the beginning of \cite[Proof of Proposition 2.3, first paragraph]{MatDer}.
\end{proof}

Now we observe that the map $f\to f_{\psi}$ is a $R(G_n)$-module endomorphism of $\sm(G_n)$, hence if the $G_n$ module $V:=\langle R(G_n)f \rangle$ generated by $f$ has finite length, so does $V_\psi:=\langle R(G_n)f_\psi \rangle$. Let us set 
\[A_{n-1,1}:=A_{M_{n-1,1}}=Z(M_{n-1,1}).\] We recall from \cite[Lemma 2.1]{MatRT} for example, that the $A_{n-1,1}$-submodule $\langle \overline{R(A_{n-1,1})f_\psi} \rangle$ generated by the image $\overline{f_\psi}$ of $f_\psi$ in $(V_{\psi})_{P_{n-1,1}}$ is finite dimensional. A consequence of Lemma \ref{lm as whit2} is the following asymptotic expansion:

\begin{lem}\label{lm as whit3}
Let $f$ belong to $\sm(G_n)$ and suppose that the $G_n$ module $V:=\langle R(G_n)f \rangle$ generated by $f$ has finite length. Choose $f_1,\dots,f_r$ in $V$ such that the projections of the $f_{i,\psi}$ in $(V_{\psi})_{P_{n-1,1}}$ form a basis of $\langle \overline{R(A_{n-1,1})f_\psi} \rangle$. Then for each $\chi \in \CE(A_{n-1,1},V_{P_{n-1,1}})$, there exists polynomials $P_{\chi,1},\dots,P_{\chi,r}$ in $\BC[X]$, and there exists an integer $a$ such that 
for any $t\in E^\times$ with $|t|_E\leq q_E^a$ and any $p\in \diag(P_{n-1},1)$, one has 
\[f_\psi(z_{n-1}(t)p)=\sum_{\chi \in \CE(A_{n-1,1},V_{P_{n-1,1}})}\sum_{i=1}^r \chi(z_{n-1}(t))P_{\chi,i}(v_E(t))f_{i,\psi}(p).\]
\end{lem}
\begin{proof}
By \cite[Proposition 2.8]{MatRT}, for each $\chi \in \CE(A_{n-1,1},V_{P_{n-1,1}})$, there exists polynomials $P_{\chi,1},\dots,P_{\chi,r}$ in $\BC[X]$ such that for any $t\in E^\times$:
\begin{equation}\label{eq sum exp} R(z_{n-1}(t))\overline{f_\psi}=\sum_{\chi \in \CE(A_{n-1,1},(V_{\psi})_{P_{n-1,1}})}\sum_{i=1}^r \chi(z_{n-1}(t))P_{\chi,i}(v_E(t))\overline{f_{i.\psi}}. 
\end{equation} Now we observe that because $V_\psi$ is a $G_n$-quotient of $V$, we have the inclsuion of exponent sets 
\[\CE(A_{n-1,1},(V_{\psi})_{P_{n-1,1}})\subseteq  \CE(A_{n-1,1},V_{P_{n-1,1}}).\] hence we can as well sum over $\CE(A_{n-1,1},V_{P_{n-1,1}})$ in Equation \eqref{eq sum exp} by taking the polynomials $P_{\chi,i}$ to be zero if $\chi\notin \CE(A_{n-1,1},(V_{\psi})_{P_{n-1,1}})$ (which actually does not happen, but we won't discuss this detail). The statement now follows from Lemma \ref{lm as whit2}. 
\end{proof}

We will need the following result.

\begin{prop}\label{prop main}
Suppose that $f\in \CA_2(G_n)$, then the integral \[\int_{Z_{n-1}^{\theta}} f_{\psi}(z)dz\] is absolutely convergent, and moreover the map 
$P_f:Z_{n-1}^{\theta}\backslash G_{n-1}\to \BC$ defined by 
\[P_f(\overline{g})=\int_{Z_{n-1}^\theta} f_{\psi}(z\diag(g,1))dz\] belongs to $\CC(Z_{n-1}^{\theta}\backslash G_{n-1})$. 
\end{prop}
\begin{proof}
If $f\in \CA_2(G_n)$, then all the central exponents in $\CE(A_{n-1,1},V_{P_{n-1,1}})$ are positive. Now take $a_0$ as in Lemma \ref{lm as whit3} and 
$b_0\geq a_0$ as in Lemma \ref{lm as whit1}. Then 
\[\int_{Z_{n-1}^{\theta}} f_{\psi}(z)dz=\]
\[
\int_{t\in F^\times, \ |t|_E\leq q_E^{a_0}} f_{\psi}(z_{n-1}(t))dt+ \int_{t\in F^\times, \ q_E^{a_0}< |t|_E\leq q_E^{b_0}} f_{\psi}(z_{n-1}(t))dt.\]
The first summand converges absolutely thanks to Lemma \ref{lm as whit3}, because the characters $\chi(z_{n-1}(\ ))$ for $\chi\in \CE(A_{n-1,1},V_{P_{n-1,1}})$ are positive, and the second summand is actually a finite sum. Then we observe that if $K$ is a compact open subgroup of $G_n$ and $f\in \CA_2(G_n)_{K}$, one has \[P_f\in \sm(Z_{n-1}^{\theta}\backslash G_{n-1})_{\overline{K\cap G_{n-1}}}\] by straightforward change of variables. 
Now to prove that $P_f$ belongs to $\CC(Z_{n-1}^{\theta}\backslash G_{n-1})$, thanks to the Cartan decomposition of $Z_{n-1}^{\theta}\backslash G_{n-1}$, it is enough to obtain the Harish-Chandra Schwartz majorizations on \[\overline{t(E^\times,\dots,E^\times,1)}\leq Z_{n-1}^{\theta}\backslash G_{n-1}.\] By the Cartan decomposition again, this time for $G_{n-2}$, it is thus enough to obtain those majorizations on $\overline{\diag(G_{n-2}^1,1)}$, and we claim that it is actually enough to obtain them on $\overline{\diag(G_{n-2}^1,1)}$: indeed $\diag(G_{n-2}^1,1)Z_{n-1}$ is the inverse image of $E^\times(n-1)$ by the determinant map inside the group $\diag(G_{n-2},1)Z_{n-1}$, hence of finite index inside it. Now take $a_0\leq b_0$ for $f$ as in Lemmata \ref{lm as whit1} and \ref{lm as whit3}, then for any $g\in G_{n-2}$ we have 
\[\int_{Z_{n-1}^{\theta}} f_{\psi}(zg)dz=\]
\[\sum_i(\int_{t\in F^\times, \ |t|_E\leq q_E^{a_0}}  \sum_{\chi} \chi(z_{n-1}(t))P_{\chi,i}(v_E(t))dt)f_{i,\psi}(g)\]
\[+ \int_{t\in F^\times, \ q_E^{a_0}< |t|_E\leq q_E^{b_0}} f_{\psi}(z_{n-1}(t)g)dt,\] where the second integral is a finite sum of 
of the form \[\sum_k f_{\psi}(z_{n-1}(t_k)g )\] with $t_k$ independent of $g\in G_{n-2}$. But then by Corollary \ref{cor Ok lm cons}, we deduce that $\Res_{\diag(G_{n-2},1)}(P_f) $ can be expressed as a finite sum of right translates of $\Res_{\diag(G_{n-2},I_2)}(f_i)$ and $\Res_{\diag(G_{n-2},I_2)}(f)$, so that the result now follows from Corollary \ref{cor res schw}.
\end{proof}

\section{The relative converse theorem}\label{sec RLC}

We now follow Ok's arguments closely, but we extend them to the setting of Harish-Chandra Schwartz spaces. First we need a technical result which is not needed in Ok's work, as he deals with compactly supported functions. We recall that $\psi:E\to \BC^\times$ is a non trivial additive character, trivial on $F^\times$ when $[E:F]=2$.

\subsection{An inversion formula of Ok for Whittaker functions}\label{sec inv Four Ok}

Here $G_n$ and $\psi$ are fixed as before. By applying Lemma \ref{lm unfolding 1} repeatedly, one deduces the following inversion formula, which is \cite[Main Lemma I]{Ok} when $[E:F]=2$, and holds with the exact same proof when $E=F$. 

\begin{thm}[Ok's inversion formula for Whittaker functions]\label{Ok Whittaker inversion formula}
Let $K$ be a compact open subgroup of $G_n$. There exists of an increasing exhaustive family $(X_{n}^k)_{k\geq 0}$ of compact open subset of $N_{n}/ N_{n}^\theta $ such that for any $\pi\in \ugdist(G_{n})$ and any $W\in \CW(\pi,\psi^{-1})^K$, there is $c(\lambda_{\pi})\in \BC^\times$ such that 
\begin{equation}\label{eq ok lm1} \int_{X_{n}^k}\lambda_{\pi}(\pi(u^{-1})W)\psi^{-1}(u)du=c(\lambda_{\pi})W(I_{n})\end{equation} for all $k\geq 0$. 
\end{thm} 

\subsection{An absolutely convergent integral}\label{sec abs cv}

Let $f\in \CC(Z_n^\theta\backslash G_n)$. Proposition \cite[II.4.5]{WPL} shows that  
\[\int_{N_n} |f(ut)|du\prec \delta_{B_n}^{1/2}(t)(1+\sigma(t))^{-d}\] on $Z_n^\theta\backslash T_n$ for all $d\in \BR$. This has the following consequences.  

\begin{lem}\label{lm stk}
Suppose that $f\in \CC(Z_n^\theta\backslash G_n)$, then the double integral 
\[\int_{Z_n^\theta N_n^\theta\backslash H_n}\int_{N_n}|f(uh)|dudh\] is absolutely convergent and equal to 
\[\int_{N_n/N_n^\theta}\int_{Z_n^\theta \backslash H_n}|f(uh)|dhdu.\] 
\end{lem}
\begin{proof}
We write \[\int_{Z_n^\theta N_n^\theta\backslash H_n}\int_{N_n}|f(uh)|dudh=\int_{Z_n^\theta \backslash T_n}\int_{K_n^\theta}\int_{N_n}|f(utk)|du\delta_{B_n}^{-1/2}(t)dkdt ,\] which by smoothness of $f$ reduces the problem to proving the convergence of 
\[\int_{Z_n^\theta \backslash T_n^\theta}\int_{N_n}|f(ut )|du\delta_{B_n}^{-1/2}(t) dt.\] This integral is majorized by a positive multiple of \[\int_{Z_n^\theta \backslash T_n^\theta}(1+\sigma(t))^{-d}dt\] for any positive $d$, by the observation before the lemma. However this latter integral is convergent for $d$ large enough by \cite[Lemme II.1.5]{WPL}. The equality with the second integral in the statement easily follows from the first equality in the proof of the Lemma.  
\end{proof}

A consequence of Lemma \ref{lm stk} is the following proposition.

\begin{prop}\label{prop stk}
Let $f\in \CC(Z_n^\theta\backslash G_n)$, the integral \[I_\psi(f)=\int_{N_n/N_n^\theta}\int_{Z_n^\theta \backslash H_n}f(uh)\psi^{-1}(u)dhdu\] is absolutely convergent.
\end{prop}

\begin{cor}\label{cor stk}
Let $(X_n^k)_{k\geq 0}$ be and increasing family of compact open subsets which exhaust $N_n/N_n^\theta$, then 
\[I_{X_n^k,\psi}(f)=\int_{X_n^k}\int_{Z_n^\theta \backslash H_n}f(uh)\psi^{-1}(u)dudh\] converges to $I_\psi(f)$. Moreover 
the function \[J_{X_n^k,\psi}f:g\to \int_{X_n^k}\int_{Z_n^\theta \backslash H_n}f(uhg)\psi(u)dudh\] belongs to 
the Harish-Chandra Schwartz space $\CC(H_n\backslash G_n)$.
\end{cor}
\begin{proof}
In view of Proposition Proposition \ref{prop stk}, the first assertion follows from the dominated convergence theorem. The second follows from 
Proposition \ref{prop CZ}. 
\end{proof}

\subsection{The orthogonal of tempered distinguished Whittaker functions}

We recall from Lemma \ref{lm proj to WHCS space} that for $f\in \CC(Z_{n-1}^\theta\backslash G_{n-1})$, we defined  $W_f\in \CC(N_{n-1}Z_{n-1}^\theta\backslash G_{n-1},\psi)$. In particular, for any $\pi\in \temp(Z_{n-1}^\theta\backslash G_{n-1})$ and any $W\in \CW(\pi,\psi^{-1})$, the integral 
\[\int_{N_{n-1}Z_{n-1}^\theta\backslash G_{n-1}} W_f(g)W(g)dg\] is absolutely convergent by \cite[Proposition 3.2]{DWhit}. The following lemma is a consequence of Theorem \ref{Ok Whittaker inversion formula}, Theorem \ref{thm Bern} together with the extra information on the support of the Plancherel measure of $L^2(H_{n-1}\backslash G_{n-1})$ given by Section \ref{sec TP}. It is a generalization of \cite[Lemma 11.1.2]{Ok}. 

\begin{lem}\label{lm ok orth}
Let $f\in \CC(Z_{n-1}^\theta\backslash G_{n-1})$. Suppose that for any $\pi\in \tdist(G_{n-1})$, and any $W\in \CW(\pi,\psi^{-1})$, we have 
\[\int_{N_{n-1}Z_{n-1}^\theta\backslash G_{n-1}} W_f(g)W(g)dg,\]
 Then 
\[\int_{Z_{n-1}^\theta\backslash H_{n-1}} W_f(h)dh=0.\]
\end{lem}
\begin{proof}
Take $K$ a compact open subgroup of $G_{n-1}$ such that $f\in \CC(Z_{n-1}^\theta\backslash G_{n-1})_{\overline{K}}$. 
For each $\pi\in \Htemp(\frac{G_{n-1}}{Z_{n-1}^\theta})$, we fix $\lambda_\pi$ in $\Hom_{H_{n-1}}(\CW(\pi,\psi^{-1}),\BC)$ such that Theorem \ref{thm Bern} applies with this family to the pair $(\frac{G_{n-1}}{Z_{n-1}^\theta},\frac{H_{n-1}}{Z_{n-1}^\theta})$. Now Theorem \ref{Ok Whittaker inversion formula} provides an increasing exhaustive family $(X_{n-1}^k)_{k\geq 0}$ of compact open subset of $N_{n-1}/ N_{n-1}^\theta $ such that for any $\pi\in \ugdist(G_{n-1})$ and any $W\in \CW(\pi,\psi^{-1})^K$, there exists $c(\lambda_{\pi})\in \BC^\times$ such that 
\begin{equation}\label{eq ok lm1} \int_{X_{n-1}^k}\lambda_{\pi}(\pi(u^{-1})W)\psi^{-1}(u)du=c(\lambda_{\pi})W(I_{n-1})\end{equation} for all $k\geq 0$. 
We then recall from Corollary \ref{cor stk} that we defined a map \[J_{X_{n-1}^k,\psi}f \in \CC(\frac{H_{n-1}}{Z_{n-1}^\theta}\backslash \frac{G_{n-1}}{Z_{n-1}^\theta})^K.\] Let's apply the Fourier inversion formula of Theorem \ref{thm Bern} to $J_{X_{n-1}^k,\psi}f$, but observe that according to Section \ref{sec TP}, the Plancherel measure of $L^2(\frac{H_{n-1}}{Z_{n-1}^\theta}\backslash \frac{G_{n-1}}{Z_{n-1}^\theta})$ is supported on $\temp(Z_{n-1}^\theta\backslash G_{n-1})$ so that we can integrate only on the tempered members of 
$\Htemp(\frac{G_{n-1}}{Z_{n-1}^\theta})$ in the Fourier inversion formula of Theorem \ref{thm Bern}. In such a situation, 
by Fubini's theorem, Equations \eqref{eq sakz} and \eqref{eq ok lm1}, and simple integration in stages and change of variables, we obtain for any tempered $\pi\in \Htemp(\frac{G_{n-1}}{Z_{n-1}^\theta})$ and any $W\in \CW(\pi,\psi{-1})^K$ the following equalities on the right hand side of the inversion formula:
\[\langle J_{X_{n-1}^k,\psi}f, W \rangle_{\lambda_\pi} = \int_{Z_{n-1}^\theta \backslash G_{n-1}}(\int_{X_{n-1}^k}f(ug)\psi^{-1}(u)du)\lambda_\pi(\pi(g)W)dg\]
 \[= \int_{X_{n-1}^k}\int_{Z_{n-1}^\theta \backslash G_{n-1}}f(ug)\lambda_\pi(\pi(g)W)dg\psi^{-1}(u)du\]
\[=\int_{X_{n-1}^k}\int_{Z_{n-1}^\theta \backslash G_{n-1}}f(g)\lambda_\pi(\pi(u^{-1})\pi(g)W))dg\psi^{-1}(u)du\]
\[=
 \int_{X_{n-1}^k}\lambda_\pi(\pi(u^{-1})\pi(f)W))\psi^{-1}(u)du=c(\lambda_\pi)(\pi(f)W)(I_{n-1})\]
\[=\int_{Z_{n-1}^\theta \backslash G_{n-1}}f(g)W(g)dg=\int_{Z_{n-1}^\theta N_{n-1}^\theta\backslash G_{n-1}}W_f(g)W(g)dg=0.\]
Considering the left hand side of the Fourier inversion formula, this gives for all $k\geq 0$:
\[0=J_{X_{n-1}^k,\psi}f(\overline{I_{n-1}})=I_{X_{n-1}^k,\psi}(f).\] By Corollary \ref{cor stk} and Lemma \ref{lm stk}, this finally gives  
\[0=I_\psi(f)=\int_{Z_{n-1}^\theta \backslash H_{n-1}}\int_{N_{n-1}/N_{n-1}^\theta}f(uh)\psi^{-1}(u)dudh=\]
\[\int_{Z_{n-1}^\theta N_{n-1}^\theta \backslash H_{n-1}}\int_{N_{n-1}}f(uh)\psi^{-1}(u)dudh=\int_{Z_{n-1}^\theta N_{n-1}^\theta \backslash H_{n-1}}W_f(u)du.\]
\end{proof}
 
\subsection{The proof of the main result}

Before moving on to the relative converse theorem, we observe that its statement is non empty thanks to Corollary \ref{cor trivial gamma}, which states that $\gamma(\pi,\pi',\psi)=1$ whenever $\pi\in \udist(G_{an})$ and $\pi'\in \tdist(G_{a(n-1)})$.

In view of Lemma \ref{lm ok orth}, the relative converse theorem easily follows from the $(n,n-1)$ Rankin-Selberg equation of \cite{JPSS}, with an extra observation when $E=F$. We denote by $w_n$ the antidiagonal matrix of $G_n$ with ones on the antidiagonal, and for $W\in \Ind_{N_n}^{G_n}(\psi)$, we set 
\[\widetilde{W}(g)=W(w_n{}^t\!g^{-1}).\] We recall the consequence of the functional equation that we need.

\begin{prop}\label{prop fct eq}
Let $\pi\in \usq(G_n)$, and $\pi'\in \temp(G_{n-1})$, then for any $(W,W')\in \CW(\pi,\psi)\times \CW(\pi',\psi^{-1})$ we have 
\[\int_{N_{n-1}\backslash G_{n-1}} \widetilde{W}(\diag(g,1))\widetilde{W'}(g)dg=\]
\[\gamma(1/2,\pi,\pi',\psi)\int_{N_{n-1}\backslash G_{n-1}} W(\diag(g),1)W'(g)dg.\] Here both sides of the functional equation are absolutely convergent and $\gamma(1/2,\pi,\pi',\psi)\neq 0$.
\end{prop}
\begin{proof}
The absolute convergence of both sides follows for example from the asymptotics of Whittaker functions given in \cite{LMas} and \cite[Section 3.4]{DWhit}, or \cite{MatDer}. The claim on the gamma factor follows from the fact that $L(1/2,\pi,\pi')\in \BC^\times$ for any couple of generic unitary representations (see for example the properties of $L$ factors of \cite{JPSS} recalled in \cite{MO}). 
\end{proof}

We will also use the following fact from \cite[Proposition 2.3]{LM} and the discussion after it. 

\begin{lem}\label{lm LM}
Let $\pi\in \usq(G_n)$. For any matrix coefficient $\Phi$ of $\pi$, the function 
\[W(\Phi):g\in G_{n}\to \int_{N_n}\Phi(ug)\psi^{-1}(u)du\] is defined by absolutely convergent integrals. Moreover it belongs to $\CW(\pi,\psi)$, and any Whittaker function in $\CW(\pi,\psi)$ is of this form. 
\end{lem}

Here is our main result, which is optimal in the sense explained after its proof.

\begin{thm}\label{thm main}
Let $\pi\in \usq(G_{an})$ such that $\gamma(\pi,\pi',\psi)=1$ for any $\pi'\in \tdist(G_{a(n-1)})$, then $\pi\in \usqdist(G_{an})$. 
\end{thm}
\begin{proof}
First we observe that when $E=F$, using the Bernstein-Zelevinsky product notation, the representation $\pi'\times 1$ belongs to $\tdist(G_{an-1})$ according to \cite[Proposition 3.8]{MatCrelle} and \cite[Corollary 3.6]{Yunitary}, and that any representation in $\tdist(G_{an-1})$ is actually of this form thanks to \cite[3.13]{MatCrelle} and \cite[Corollary 3.6]{Yunitary} again. We also observe that by multiplicativity of gamma factors: $\gamma(1/2,\pi,\pi'\times 1,\psi)= \gamma(1/2,\pi,\psi)$. Hence we put $c_0=1$ when $[E:F]=2$ and $c_0=\gamma(\pi,\psi)\in \BC^\times$ when $E=F$. 
We fix $W\in \CW(\pi,\psi)$ and write it under the form $W(\Phi)$ thanks to Lemma \ref{lm LM}. For $g\in G_{an}$ we then put 
\[f(g):=c_0\Phi(g)-\Phi(w_{an}\diag(w_{an-1},1)g),\] so that $f\in \CA_2(G_{an})$. 
Using the notations of Lemma \ref{lm proj to WHCS space} and Proposition \ref{prop main}, we observe that 
\[W_{P_f}(g)=\int_{Z_{an-1}^\theta} c_0W(\diag(zg,1))-W(w_{an}\diag(w_{an-1}zg,1))dz\] for all $g\in G_{n-1}$. 
Rewriting the functional equation of Proposition \ref{prop fct eq}, we obtain for any $\pi''\in \tdist(G_{an-1})$ and any 
$W''\in \CW(\pi'',\psi^{-1})$: 
\[0=\int_{N_{an-1}\backslash G_{an-1}} (c_0 W(\diag(g,1))-W(w_{an}\diag(w_{an-1}g,1))) W''(g)dg\] 
\[=\int_{N_{an-1}Z_{an-1}^\theta\backslash G_{an-1}}W_{P_f}(g) W'(g)dg.\]

By Lemma \ref{lm ok orth}, this implies that 
\[\int_{N_{an-1}^\theta Z_{an-1}^\theta\backslash H_{an-1}}W_{P_f}(h)=0,\] which we rewrite 
\[c_0\int_{N_{an-1}^\theta\backslash H_{an-1}} W(\diag(h,1))= \int_{N_{an-1}^\theta\backslash H_{an-1}} \widetilde{W}(\diag(h,1)).\]
Hence we have two linear forms on $\CW(\pi,\psi)$
\[\lambda_{\pi}:W\to \int_{N_{an-1}^\theta\backslash H_{an-1}} W(\diag(h,1)) \] and 
\[\mu_{\pi}:W\to \int_{N_{an-1}^\theta\backslash H_{an-1}} \widetilde{W}(\diag(h,1))\] which are equal up to a nonzero scalar, but the first is $P_{an}^\theta$-invariant, whereas the second 
is ${}^t\! P_{an}^\theta$-invariant. Hence $\lambda_{\pi}$ is $H_{an}$-invariant because $H_{an}$ is generated by $P_{an}^\theta$ and its transpose. It is moreover nonzero by the theory of Bernstein-Zelevinsky derivatives, as recalled in the proof of Lemma \ref{lm BFmat}  
\end{proof}

\begin{rem}\label{rem final 1}
Originally Ok expected his result to extend to tempered representations. In \cite[Example 5.6]{MO}, by giving a fake counter-example, we wrongly claimed that Ok's result can at most be extended to discrete series representations. This was observed by Beuzart-Plessis, who took the time to give a real counter-example to \cite[Example 5.6]{MO}, and who also observed that one can extend Ok's result to all tempered representations, as well as Archimedean local fields, by following the steps of this paper and using results from \cite{BPplanch}. More generally Theorem \ref{thm main} can be extended to all tempered representations. We shall come back to it soon. 
\end{rem}

\begin{rem}
Theorem \ref{thm main} can be restated in terms of representations of the Weil-Deligne group, using the Artin-Deligne-Langlands gamma factors as follows.
\begin{enumerate}
\item When $[E:F]=2$, replace in its statement $\pi$ by an irreducible representation $\phi$ of dimension $n\geq 2$ of $W_F\times \SL_2(\BC)$, and make the twisting representation vary amongst the semi-simple and bounded conjugate-orthogonal representations of $W_F\times \SL_2(\BC)$ of dimension $n-1$. The conclusion becomes that $\phi$ is conjugate-orthogonal.
\item When $E=F$, replace in its statement $\pi$ by an irreducible representation $\phi$ of dimension $2n\geq 4$ of $W_F\times \SL_2(\BC)$, and make the twisting representation vary amongst the semi-simple and bounded symplectic representations of $W_F\times \SL_2(\BC)$ of dimension $2n-2$. The conclusion becomes that $\phi$ is symplectic.
\end{enumerate}
\end{rem}

\begin{rem}\label{rem final 2}
We feel that one could perhaps replace gamma factors by epsilon factors in Theorem \ref{thm main}. It is obvious if one assumes that $\pi$ is cuspidal, and not hard to check if one assumes a priori that the discrete series $\pi$ satisfies $\pi^\vee\simeq \pi^\theta$. In particular if $E=F$ and $\pi$ is cuspidal, or equivalently if $\phi_\pi$ is an irreducible representation of $W_F$, Theorem \ref{thm main} can be thought as a converse of \cite[Proposition 2.1]{PR}.
\end{rem}

	\bibliographystyle{alphanum}
	\bibliography{references}
	
\end{document}